\def\mystyle{}

\if\mystyle l
\documentclass[landscape]{amsart}
\else
\documentclass{amsart}
\fi

\usepackage{begnac}
\usepackage{stmaryrd}

\begin{document}

\title{Grey subsets of Polish spaces}

\author{Itaï \textsc{Ben Yaacov}}

\address{Itaï \textsc{Ben Yaacov} \\
  Université Claude Bernard -- Lyon 1 \\
  Institut Camille Jordan, CNRS UMR 5208 \\
  43 boulevard du 11 novembre 1918 \\
  69622 Villeurbanne Cedex \\
  France}

\urladdr{\url{http://math.univ-lyon1.fr/~begnac/}}

\author{Julien \textsc{Melleray}}

\address{Julien \textsc{Melleray} \\
  Université Claude Bernard -- Lyon 1 \\
  Institut Camille Jordan, CNRS UMR 5208 \\
  43 boulevard du 11 novembre 1918 \\
  69622 Villeurbanne Cedex \\
  France}

\urladdr{\url{http://math.univ-lyon1.fr/~melleray/}}

\thanks{Research supported by the Institut Universitaire de France and ANR contract GruPoLoCo (ANR-11-JS01-008).}

\thanks{We wish to thank Alexander Berenstein for several helpful discussions as well as Todor Tsankov for valuable comments.}

\svnInfo $Id: SmallDensity.tex 1903 2014-09-04 11:47:41Z begnac $
\thanks{\textit{Revision} {\svnInfoRevision} \textit{of} \today}

%\date{\today}

\begin{abstract}
  We develop the basics of an analogue of descriptive set theory for functions on a Polish space $X$.
  We use this to define a version of the small index property in the context of Polish topometric groups, and show that Polish topometric groups with ample generics have this property. We also extend classical theorems of Effros and Hausdorff to the topometric context.
\end{abstract}

\maketitle

\tableofcontents

\section*{Introduction}

Work on this paper initially started as a follow-up to \cite{BenYaacov-Berenstein-Melleray:TopometricGroups}, in which  we introduced the notion of a \emph{Polish topometric group}, and defined a notion of \emph{ample generics} in that context.

Recall that a \emph{Polish metric structure} $\mathcal M$ is a complete, separable metric space $(M,d)$, along with a family $(R_i)_{i \in I}$ such that each $R_i$ is a uniformly continuous map from some $M^{n_i}$ to $\bR$.
(One may also allow symbols $f_j$ for uniformly continuous functions from $M^{m_j}$ to $M$.
However, any such function is equally well coded by $(\bar x,y) \mapsto d\bigl( f_j(\bar x),y \bigr) \in \bR$, so no generality is gained.
In particular, a constant, or zero-ary function, $c \in M$, is coded by $d(\cdot,c)$.)
The automorphism group $\Aut(\cM)$ is made up of all the isometries of $(M,d)$ which preserve all the relations $R_i$.
When endowed with the pointwise convergence topology $\tau$, $\Aut(\mathcal M)$ is a Polish group.
One can also consider the metric of uniform convergence $\partial$ defined by
\begin{gather*}
  \partial(g,h)= \sup \bigl\{ d(gm,hm) : m \in M \bigr\}.
\end{gather*}
This metric $\partial$ is complete, bi-invariant, and in general not separable.
It also refines $\tau$, and is $\tau$-lower semi-continuous, i.e., the sets $\{(g,h) : \partial(g,h) \le r\}$ are closed for all $r$.

With this paradigm in mind, we define a \emph{Polish topometric group} as any triplet $(G,\tau,\partial)$, where $(G,\tau)$ is a Polish group and $\partial$ possesses the properties cited above.
Automorphism groups of Polish metric structures, and therefore Polish topometric groups, are ubiquitous in analysis.
Of particular interest to us are the unitary group $\mathcal U(\ell_2)$ of a complex separable infinite-dimensional Hilbert space, the group $\Aut(\lambda)$ of  all measure-preserving isomorphism of a standard probability space, and the isometry group of the Urysohn space $\bU$.
The first two structures in that list in particular have natural, and well-studied, topometric structures: for instance, the topometric structure on $\mathcal U(\ell_2)$ is given by the strong operator topology and the operator norm (since $\ell_2$ is unbounded, the uniform convergence metric is calculated on its unit ball).
The notion of ample generics for Polish topometric groups involves an interplay between the topological and metric structures, and it was shown that all three examples above have ample generics.
This was applied to show that $\Aut([0,1],\lambda)$, endowed with its usual Polish topology, has the automatic continuity property, i.e., any group homomorphism with source $\Aut([0,1],\lambda)$ and taking values in a separable topological group must be continuous. This automatic continuity theorem was given a shorter proof, which also involves the use of the uniform metric and the pointwise convergence topology, by Malicki in \cite{Malicki:LebesgueAutomorphismGroup}, where the small index property of $\Aut([0,1],\lambda)$ is also investigated.
Soon after \cite{BenYaacov-Berenstein-Melleray:TopometricGroups} Tsankov \cite{Tsankov:UnitaryAutomaticContinuity} used the existence of ample generics for $\mathcal U(\ell_2)$ to show that it, too, has the automatic continuity property. We should note here that, very recently, Sabok \cite{Sabok:AutomaticContinuity} gave new proofs of automatic continuity theorems for $\Aut([0,1],\lambda)$ and $\mathcal U(\ell_2)$ which do not involve topometric structure; his technique also allowed him to prove that the isometry group of the Urysohn space satisfies the automatic continuity property.

Further study of topometric groups, and of ample generics in this context, leads to an interesting phenomenon, or obstacle, namely that results which would, in the usual Polish group context, be about \emph{sets}, are naturally formulated in the topometric context as results about \emph{functions}.
The reason behind this phenomenon is that in the presence of a metric, one no longer asks whether two things are equal or not, but rather, how far apart they are.
Thus, for example, where one would consider a set of the form $\bigl\{ x: f(x) = g(x) \bigr\}$, we often find ourselves considering the function $x \mapsto d\bigl( f(x), g(x) \bigr)$ (or $\partial\bigl( f(x), g(x) \bigr)$, depending on what the objects in question are).
Similarly, a point $x$ is naturally represented by the function $d(x,\cdot)$.
It is not always immediately clear, however, how statements should be translated from the context of sets to that of functions.

In the beginning of this paper we develop the basics of descriptive set theory in that context, where sets are replaced by \emph{grey sets}, i.e., functions with values in $[0,+\infty]$.
A rule of thumb which every such extension should satisfy is that, in the particular case of functions only taking the values $0$ and $+\infty$, it should just boil down to the usual statement regarding sets (where a set $A$ is represented by its zero-indicator $\bzero_A$, see below).
The first (and easy) task is to figure out the analogues of open sets, closed sets, meagre sets, and so on, after which we extend a few basic results of descriptive set theory to the grey setting.
Having thus argued that such an approach can be made to work, we focus on the case of Polish groups, introducing in particular a semi-group structure on grey sets.
Using this semi-group structure, one can see that the analogue of a subgroup in our context (i.e, a \emph{grey subgroup}) is a \emph{semi-norm} on $G$, i.e., a map $H$ such that
\begin{itemize}
\item $H(1)=0$,
\item $\forall g \in G \ H(g)=H(g^{-1})$,
\item $\forall g,g' \in G \ H(gg') \le H(g)+H(g')$.
\end{itemize}

Such a function $H$ naturally defines a left-invariant pseudo-metric $d_H$ on $G$, defined by $d_H(g,h)=H(g^{-1}h)$, and the \emph{index} of $H$ is simply the density character of the metric space obtained when identifying points $g,h$ such that $d_H(g,h)=0$.
A ``grey'' analogue of the small index property would then be: whenever $H$ is a left-invariant pseudo-metric on $G$ with density character $< 2^{\aleph_0}$, $H$ must be continuous with respect to $\tau$.
This version implies immediately the classical version of the small index property, as well as automatic continuity, and in fact can be shown to be equivalent to the latter.
A topometric version thereof must take $\partial$ into account (or else be too strong), and one is led to the following statement, which is one of our main results.

\begin{quotethm}{thm:SmallDensity}
  Assume that $(G,\tau,\partial)$ is a Polish topometric group with ample generics, and that $H$ is a semi-norm on $G$ which is Baire-measurable with respect to $\partial$ and has density character $<2^{\aleph_0}$.
  Then $H$ is continuous with respect to $\tau$.
\end{quotethm}

We say that a topometric group satisfying the conclusion of the above theorem has the \emph{small density property}.
It is important to note that the above result can also be obtained from one of the main results of \cite{BenYaacov-Berenstein-Melleray:TopometricGroups} and in that sense is not new.
That said, the version presented here implies the said result of \cite{BenYaacov-Berenstein-Melleray:TopometricGroups} in a trivial manner, and moreover, the ``grey approach'' renders possible a more streamlined proof.
We therefore contend that the grey approach is a better presentation of the theory.

We then use our techniques to establish a topometric version of a theorem due to Effros, which states (among other equivalent conditions) that, if $G$ is a Polish group acting continuously on a Polish space $X$ and $x \in X$ has a dense orbit, then $G  x$ is co-meagre if and only if the map $g \mapsto g  x$ is open from $G$ to $G  x$.
One of the implications in this theorem is a direct consequence of Hausdorff's theorem stating that an open, metrisable image of a Polish space is Polish.
After obtaining a topometric version of Hausdorff's theorem (\autoref{thm:Hausdorff} below), we establish the following result, the proof of which requires the machinery of grey sets.

\begin{quotethm}{thm:EffrosTopometric}
  Let $(X,\tau,\partial)$ be a Polish topometric space and $G$ a Polish group acting continuously on $X$ by $\tau$-homeomorphisms and $\partial$-isometries.
  Assume that $(U)_{\partial<r} =\{x \in X : \exists u \in U \ \partial(x,u) < r \}$ is open in $X$ for any open subset $U$ of $G$ and any $r>0$, and that $x \in X$ is such that $G  x$ is dense. Then the following conditions are equivalent:
  \begin{enumerate}
  \item $\overline{G  x}^{\partial}$ is $G_{\delta}$.
  \item $\overline{G  x}^{\partial}$ is co-meagre.
  \item For any open subset $U$ of $G$ and any $r >0$, $(U  x)_{\partial<r}$ is open in $\overline{G  x}^{\partial}$.
  \item There exists $y \in \overline{G  x}^{\partial}$ such that, for any open subset $U$ of $G$ and any $r >0$, $(U  y)_{\partial<r}$ is open in $G  y$.
  \end{enumerate}
\end{quotethm}

Note that the condition that $(U)_{\partial<r}$ is open in $X$ for any open $U$ and any $r >0$ is always satisfied when $G$ is a Polish group and $X=G^n$ equipped with its natural topometric structure; hence, the previous theorem applies in particular to the case when $G$ acts on $G^n$ by diagonal conjugation.
Also, since $(U  x)_{\partial<r}=U  (\{x\})_{\partial<r}$, this result is indeed an illustration of the fact that, when moving to the topometric setting, one replaces the notion of ``point'' by that of ``point up to a small, uniform error'', which is the information contained in the distance function to the point.

\section{Grey sets}
We recall that the classical setting for descriptive set theory is that of \emph{Polish spaces}, i.e., separable metrisable topological spaces whose topology is induced by a complete metric, or more generally that of completely metrisable topological spaces.
Our aim is to ``do some topology'', or descriptive set theory, where instead of considering subsets of a topological space $X$ we consider functions on $X$, say valued in $[0,\infty]$.

\begin{dfn}
  \label{dfn:GreySet}
  By a \emph{grey subset} of $X$, denoted $A \sqsubseteq X$, we mean a function $A \colon X \to [0,\infty]$.

  If $A$ and $B$ are two grey subsets of $X$ then we say that $A \sqsubseteq B$, or $A \geq B$, if $A(x) \geq B(x)$ for all $x$.
  Consequently, we write $\bigsqcap A_i$ for $x \mapsto \sup A_i(x)$, and similarly for $\bigsqcup A_i$, $A \sqcap B$, and $A \sqcup B$.
  We say that $A \sqsubseteq^* B$ (respectively, $A =^* B$) if $A(x) \geq B(x)$ (respectively, $A(x) = B(x)$) outside a meagre set.
\end{dfn}

\begin{ntn}
  \label{ntn:ZeroIndicator}
  For any set $A \subseteq X$ we define its \emph{zero-indicator}
  \begin{gather*}
    \bzero_A(x) =
    \begin{cases}
      0 & x \in A, \\
      \infty & x \notin A.
    \end{cases}
  \end{gather*}
\end{ntn}

An ordinary set $A \subseteq X$ can be viewed as a grey set by identifying it with its zero-indicator.

\subsection{Basic topology}

We start with some basic topological notions.
These reduce fairly easily to the corresponding notions regarding ordinary sets, via the following notation.

\begin{ntn}
  Given a function $\varphi\colon X \to [-\infty,\infty]$ and $r \in \bR$, we define $\varphi_{<r} = \bigl\{ x: \varphi(x) < r \bigr\}$, and similarly for $\varphi_{\leq r}$, etc.
\end{ntn}

Recall that a function $\varphi\colon X \to [-\infty,\infty]$ on a topological space $X$ is \emph{upper (respectively, lower) semi-continuous} if the set $\varphi_{<r}$ (respectively, $\varphi_{>r}$) is open for all $r \in \bR$.

% \begin{lem}
%   \label{lem:InfUSC}
%   Let $\Phi$ be a family of upper semi-continuous functions on a topological space $X$.
%   Then $\inf \Phi \colon x \mapsto \inf \{f(x) \colon f \in \Phi\}$ is upper semi-continuous as well.

%   If $X$ admits a countable base then there exists a countable sub-family $\Phi_0 \subseteq \Phi$ such that $\inf \Phi = \inf \Phi_0$.
% \end{lem}
% \begin{proof}
%   The first assertion is quite standard.
%   For the second, let $\cB$ be a countable base for $X$.
%   We let $\Psi$ consist of all functions of the form $q + \bzero_U$ where $q \in \bQ$ and $U \in \cB$.
%   Then $\Psi$ is a countable family of upper semi-continuous functions on $X$, and every upper semi-continuous function $\varphi$ is the infimum of those $\psi \in \Psi$ greater than $\varphi$.

%   Now, choose $\Phi_0 \subseteq \Phi$ countable such that for each $\psi \in \Psi$, if there is $\varphi \in \Phi$ such that $\varphi \leq \psi$ then there is such $\varphi$ in $\Phi_0$ as well.
%   Then $\inf \Phi = \inf \Phi_0$.
% \end{proof}

\begin{dfn}
  \label{dfn:GreySetTopology}
  Let $A \sqsubseteq X$ be a grey set.
  \begin{enumerate}
  \item
    We say that $A$ is \emph{open} (in symbols, $A \sqsubseteq_o X$) if it is upper semi-continuous as a function.
    We say that it is \emph{closed} ($A \sqsubseteq_c X$) if it is lower semi-continuous, and \emph{clopen} if it is both, namely continuous.
  \item
    We say that $A$ is \emph{$G_\delta$} if $A_{\leq r}$ is a $G_\delta$ set for all $r$.
  \item
    We say that $A$ is \emph{meagre} if $A_{<r}$ is meagre for some $r > 0$, and that it is \emph{co-meagre} if $A_{<r}$ is co-meagre for all $r > 0$, i.e., if $A_{=0}$ is co-meagre.
  \end{enumerate}
\end{dfn}

We leave it as an easy exercise for the reader to check that a grey $A$ is $G_\delta$ if and only if it is equal to $\bigsqcap O_n$ for a countable family of open grey sets $O_n$.
Also, it is clear that if $A_{<r}$ is $G_\delta$ for all $r$ then $A$ is $G_\delta$, but the converse is false.
Indeed, let $B\colon [0,1] \rightarrow [0,1]$ assign $1/m$ to a reduced rational $n/m$ and $0$ to all irrationals, and let $A = 1-B$.
Then $A_{\leq r}$ is either finite or all of $[0,1]$, and in any case $G_\delta$, but $A_{<1} = \bQ \cap [0,1]$ is not.

\begin{conv}
  \label{conv:InfinityMinusInfinity}
  When doing arithmetic in $[0,\infty]$, let us agree that $\infty - \infty = 0$.
  Recall also that $t \dotminus s = \max (0,t-s)$, which we extend by this convention to all $t,s \in [0,\infty]$.
\end{conv}

With this convention, $A \sqsubseteq^* B$ (respectively, $A =^* B$) if and only if $B \dotminus A$ (respectively, $|A-B|$) is co-meagre.

For example, $\bzero_A$ is open, i.e., upper semi-continuous (respectively, closed, i.e., lower semi-continuous) if and only if $A$ is open (respectively, closed), $\bzero_A \sqsubseteq \bzero_B$ if and only if $A \subseteq B$, and so on.
For grey sets, we have $A \sqsubseteq X$ if and only if $A \sqsubseteq \bzero_X$, and if $A_i$ are all open (respectively, closed) then so is $\bigsqcup A_i$ (respectively, $\bigsqcap A_i$).

One may restrict grey subsets to values in $[0,1]$ (in which case we define $\bzero_A(x) = 1$ when $x \notin A$), or extend to $[-\infty,\infty]$.
Since these are isomorphic ordered sets, equipped with the order topology, choosing one or the other has essentially no effect on most of our results.

\begin{dfn}
  Let $A \sqsubseteq X$.
  \begin{enumerate}
  \item We define $A^\circ$, the \emph{interior} of $A$ to be the least u.s.c.\ $B$ greater than $A$.
  \item We define $\overline A$, the \emph{closure} of $A$ to be the greatest l.s.c.\ $B$ less than $A$.
  \end{enumerate}
\end{dfn}

One can check that
\begin{gather*}
  \bigsqcup \{ B \sqsubseteq_o X : B \sqsubseteq A \} = A^\circ \sqsubseteq A \sqsubseteq \overline A = \bigsqcap \{ B \sqsubseteq_c X : B \sqsupseteq A \},
  \\
  A^\circ(x) = \limsup_{y \to x} A(y), \qquad \overline A(x) = \liminf_{y \to x} A(y),
  \\
  (A^\circ)_{<r} = (A_{<r})^\circ, \qquad (\overline A)_{\leq r} = \overline{(A_{\leq r})}.
\end{gather*}
It follows immediately that for any family $(A_i)_{i \in I}$ of grey subsets of $X$, one has $\overline{\bigsqcap_i A_i} \sqsubseteq \bigsqcap_i \overline{A_i}$ and $\bigl( \bigsqcup_i A_i \bigr)^\circ \sqsupseteq \bigsqcup_i A_i^\circ$, and similarly $\overline{A \sqcup B} = \overline A \sqcup \overline B$, $(A \sqcap B)^\circ = A^\circ \sqcap B^\circ$.

Another notion which transfers without much pain to the grey context is Baire measurability.
Throughout the rest of this section, let $X$ denote a completely metrisable topological space (not necessarily separable).
For all the basic facts and theorems of descriptive set theory we use below, we refer the reader to Kechris \cite{Kechris:Classical}.
Whenever $A,B \subseteq X$ and $A \setminus B$ is meagre in $X$ we write $A \subseteq^* B$, and if $A \subseteq^* B \subseteq^* A$ then we write $A =^* B$.
A subset $A \subseteq X$ is called \emph{Baire-measurable} if there exists an open set $U$ such that $A =^* U$.
The family of all Baire-measurable sets forms a $\sigma$-algebra which contains all open sets and therefore all Borel sets.
We recall:

\begin{fct}[{see \cite[Theorem~8.29]{Kechris:Classical}}]
  \label{fct:BaireMeasurableSet}
  When $X$ is completely metrisable, we define
  \begin{gather*}
    U(A) = \bigcup \bigl\{ U \text{ open in } X : U \subseteq^* A \bigr\}.
  \end{gather*}
  This is an open set, it is always the case that $A \supseteq^* U(A)$, and $U(A)$ is the largest open set with this property.
  The set $A$ is then Baire-measurable if and only if $A =^* U(A)$, if and only if $A \subseteq^* U(A)$.
\end{fct}

% \begin{lem}
%   \label{lem:MeagreBoundary}
%   Let $O \sqsubseteq_o X$.
%   Then $O =^* \overline{O}$.
% \end{lem}
% \begin{proof}
%   We may replace the interval $[0,\infty]$ with $[0,1]$ (e.g., via $t \mapsto 1- e ^{-t}$), and thus we may define $\Omega = O - \overline O$.
%   Then $\Omega \sqsubseteq_o X$ as well, and we wish to show that $\Omega_{=0} = \bigcap_{\varepsilon > 0} \Omega_{<\varepsilon}$ is a dense $G_\delta$.
%   Indeed, each $\Omega_{<\varepsilon}$ is open, so let us show that it is dense.
%   Now let $U$ be an open set and $x \in U$.
%   Since $\overline O$ is closed, there exists a smaller open set $x \in V \subseteq U$ such that $\overline O\rest_V > \overline O(x) - \varepsilon$.
%   Since $\overline O(x) = \liminf_{y \rightarrow x} O(y)$, there exists $y \in V$ such that $O(y) < \overline O(x) + \varepsilon$, so $O(y) < \overline O(y) + 2\varepsilon$, so $U \cap \Omega_{2\varepsilon} \neq \emptyset$, which is enough.
% \end{proof}

The following is easy and left to the reader.
\begin{lem}
  \label{lem:GreyStarContainment}
  Let $A,B \sqsubseteq X$ be two grey sets.
  Then $A \sqsubseteq^* B$ if and only if $A_{<r} \subseteq^* B_{\leq s}$ for every rational $r<s$, if and only if $A_{<r} \subseteq^* B_{<r}$ ($A_{\leq r} \subseteq^* B_{\leq r}$) for all $r$.
\end{lem}

\begin{lem}
  \label{lem:GreyBaireMeasurable}
  Let $A$ be a grey subset of $X$, and define
  \begin{gather*}
    U(A) = \bigsqcup \bigl\{ O \sqsubseteq_o X: O \sqsubseteq^* A \bigr\}.
  \end{gather*}
  Then $U(A) \sqsubseteq^* A$, $U(A)$ is the least u.s.c.\ function (namely, $\sqsubseteq$-greatest open grey set) with this property, and the following are equivalent:
  \begin{enumerate}
  \item We have $A \sqsubseteq^* U(A)$.
  \item We have $A =^* U(A)$.
  \item There exists an open grey set $B$ such that $A =^* B$.
  \item As a function, $A$ is Baire-measurable.
  \end{enumerate}
\end{lem}
\begin{proof}
  We make free use of \autoref{fct:BaireMeasurableSet} and \autoref{lem:GreyStarContainment}.
  Let $U'(A)$ be the open grey set defined by $U'(A)_{<r} = \bigcup_{s<r} U(A_{<s})$, and let us show that $U'(A) = U(A)$.
  Indeed, for each $r$ we have $\bigcup_{s<r} U(A)_{<s} = U(A)_{<r} \subseteq^* A_{<r}$, whereby $U(A)_{<r} \subseteq U(A_{<r})$ and therefore $U(A)_{<r} \subseteq U'(A)_{<r}$, i.e., $U(A) \sqsubseteq U'(A)$.
  Conversely, we have $U'(A)_{<r} \subseteq U(A_{<r}) \subseteq^* A_{<r}$, so $U'(A) \sqsubseteq^* A$ and therefore $U'(A) \sqsubseteq U(A)$.
  Thus $U(A) = U'(A) \sqsubseteq^* A$, and it is clearly $\sqsubseteq$-greatest.
  Now,
  \begin{cycprf}
  \item[\impnnext]
    Clear.
  \item[\impnext]
    If such $B$ exists then $B_{<r} =^* A_{<r}$ for all $r$.
    It follows that $A_{<r}$ is a Baire-measurable set for all $r$, i.e., $A$ is a Baire-measurable function.
  \item[\impfirst]
    Assume that $A_{<r}$ is a Baire-measurable set for all $r$.
    Then $A_{<r} \subseteq^* U(A_{<r})$, and therefore $A_{<r} = \bigcup_{s<r} A_{<s} \subseteq^* \bigcup_{s<r} U(A_{<s}) = U(A)_{<r}$.
    It follows that $A \sqsubseteq^* U(A)$.
  \end{cycprf}
\end{proof}

Notice that if $A$ is meagre then $\inf U(A) > 0$.
If $A$ is, in addition, Baire-measurable, then by \autoref{lem:GreyBaireMeasurable} the converse holds as well.

\begin{lem}
  \label{lem:GreyGdelta}
  A grey set $B \sqsubseteq X$ is $G_\delta$ if and only if there exists a countable family of $O_n \sqsubseteq_o X$ such that $B = \bigsqcap_n O_n$.
  Moreover, if $A$ and $B$ are open or closed then $A + B$ and $A \dotminus B$ are $G_\delta$.
\end{lem}
\begin{proof}
  Assume that $B$ is $G_\delta$.
  For each rational $r > 0$, let $B_{\leq r} = \bigsqcap_n O_{r,n}$, where $O_{r,n} \subseteq X$ are open, and let $O'_{r,n} = \bzero_{O_{r,n}} + r \sqsubseteq_o X$.
  Then $B = \bigsqcap_{r,n} O'_{r,n}$.
  The converse implication is immediate.

  For the moreover part, observe that
  \begin{gather*}
    (A+B)_{\leq r} = \bigcap_{s+t > r} A_{\leq s} \cup B_{\leq t}, \qquad (A \dotminus B)_{\leq r} = \bigcap_{s \dotminus t > r} A_{\leq s} \cup B_{\geq t},
  \end{gather*}
  and similarly with any of $\leq$, $\geq$ replaced with $<$, $>$.
\end{proof}

\subsection{Relative topology}

We now turn to somewhat more involved aspects of grey set topology, which \emph{do not} reduce to the corresponding properties of sets, namely closure and interior relative to a grey superset.

\begin{dfn}
  \label{dfn:GreyRelativeClosure}
  For $A \sqsubseteq B \sqsubseteq X$, we define the \emph{relative closure} of $A$ in $B$ as $\overline A_B = \overline A \sqcap B$.
  We say that \emph{$A$ is (relatively) closed in $B$}, in symbols $A \sqsubseteq_c B$, if $A = \overline A_B$.
\end{dfn}

\begin{lem}
  \label{lem:GreyRelativeClosure}
  Let $B \sqsubseteq X$.
  \begin{enumerate}
  \item If $C \sqsubseteq_c X$ then $C \sqcap B \sqsubseteq_c B$.
  \item If $(A_i)_{i \in I}$ is a family of relatively closed grey subsets of $B$, then $\bigsqcap_i A_i$ is also closed in $B$.
  \item The grey set $\overline A_B$ is the $\sqsubseteq$-least relatively closed grey subset of $B$ containing $A$.
  \end{enumerate}
\end{lem}
\begin{proof}
  For the first item, $C \sqcap B \sqsubseteq \overline{C \sqcap B} \sqcap B \sqsubseteq \overline C \sqcap \overline B \sqcap B = C \sqcap B$.
  For the second, let $D = \bigsqcap A_i$.
  Then $D = D \sqcap B \sqsubseteq \overline D \sqcap B \sqsubseteq \bigsqcap (\overline A_i \sqcap B) = D$.
  For the third item, we have $\overline A_B \sqsubseteq_c B$ by the first item, and if $A \sqsubseteq C \sqsubseteq_c B$ then $\overline A_B \sqsubseteq \overline C_B = C$.
\end{proof}

\begin{dfn}
  For $A \sqsubseteq B \sqsubseteq X$, we say that \emph{$A$ is dense in $B$} if $\overline A_B=B$ (equivalently, if $\overline A \sqsupseteq B$, which is further equivalent to $\overline A= \overline B$).
\end{dfn}

\begin{lem}
  \label{lem:RelativeDensity}
  For any $A \sqsubseteq B \sqsubseteq X$, $A$ is dense in $B$ if, and only if, for any open $U \sqsubseteq X$, we have $\inf (B + U) \geq \inf (A + U)$.
\end{lem}
\begin{proof}
  The latter condition is easily checked to be equivalent to $\overline A \sqsupseteq \overline B$.
\end{proof}

We turn to the definition of relative interior inside a grey set, which may seem less straightforward than the definition of relative closure.
In particular, it is not obtained from the latter by mere passage to the complement (indeed, there is no notion of relative complement of a grey subset).
We leave it to the reader to check that for ordinary sets (identified with their zero-indicator functions) this gives the usual notions.
While there are other possible definitions which pass this test, and even verify \autoref{lem:GreyRelativeInterior} (possibly with the exception of the first item), it is this definition which allow us to prove the relative Baire Category Theorem below.
Let us start with a few useful technical results

\begin{lem}
  \label{lem:InequalityInteriorClosure1}
  For any $A, B \sqsubseteq X$ one has $\overline{A+B} \sqsupseteq \overline A + B^\circ$.
\end{lem}
\begin{proof}
  Recall that by \autoref{conv:InfinityMinusInfinity},
  \begin{gather*}
    t \dotminus s
    =
    \begin{cases}
      t - s & \text{if } s < t, \\
      0 & \text{otherwise}.
    \end{cases}
  \end{gather*}
  It follows that the grey set $\overline{A+B} \dotminus B^\circ$ is closed, and $\overline{A+B} \dotminus B^\circ \leq A$, i.e., $A \sqsubseteq \overline{A+B} \dotminus B^\circ$.
  Then $\overline A \sqsubseteq \overline{A+B} \dotminus B^\circ$ and therefore
  \begin{gather*}
    \overline{A+B} \sqsupseteq \overline{A+B} \dotminus B^\circ + B^\circ \sqsupseteq \overline A + B^\circ,
  \end{gather*}
  as desired.
\end{proof}

\begin{lem}
  \label{lem:InequalityInteriorClosure2}
  For any $A \sqsubseteq B \sqsubseteq X$, one has $A^\circ_B \sqsupseteq A^\circ$, or equivalently, $(A-B)^\circ + \overline B \sqsupseteq A^\circ$.
\end{lem}
\begin{proof}
  Assume first that $A$ and $B$ are bounded, and let $C = B-A + \sup A$.
  Then the desired inequality is equivalent to $\overline{C+A} \sqsupseteq \overline C + A^\circ$ which follows from \autoref{lem:InequalityInteriorClosure1}.
  For the general case, truncate $A$ and $B$ at some $M > 0$, and let $M \rightarrow \infty$.
\end{proof}

\begin{dfn}
  \label{dfn:GreyRelativeInterior}
  For $A \sqsubseteq B \sqsubseteq X$, we define the \emph{relative interior} of $A$ in $B$ as $A^\circ_B = \bigl( (A-B)^\circ + \overline B \bigr) \sqcap A$, where we recall that by convention $\infty-\infty = 0$.
  We say that \emph{$A \sqsubseteq B$ is (relatively) open in $B$}, in symbols $A \sqsubseteq_o B$, if $A = A^\circ_B$, or equivalently, if $A \sqsubseteq (A-B)^\circ + \overline B$.
\end{dfn}

\begin{lem}
  \label{lem:GreyRelativeInterior}
  Let $B \sqsubseteq X$.
  \begin{enumerate}
  \item If $A \sqsubseteq_o X$ then $A+B \sqsubseteq_o B$.
    If, in addition, $A \sqsubseteq B$ then $A \sqsubseteq_o B$.
  \item If $(A_i)_{i \in I}$ is a family of relatively open grey subsets of $B$, then $\bigsqcup_i A_i$ and $A_0 \sqcap A_1$ are also open in $B$.
  \item Assume $A \sqsubseteq B$.
    Then the grey set $A^\circ_B$ is the $\sqsubseteq$-greatest grey subset of $A$ relatively open in $B$.
    In particular, $A^\circ_B \sqsupseteq A^\circ$.
  \end{enumerate}
\end{lem}
\begin{proof}
  The first item is immediate from the definitions and \autoref{lem:InequalityInteriorClosure2}.
  For the second, let $D = \bigsqcup A_i$.
  Then $(D - B)^\circ + \overline B \sqsupseteq \bigsqcup (A_i - B)^\circ + \overline B \sqsupseteq D$.
  Similarly, $\bigl( (A_0 \sqcap A_1) - B \bigr)^\circ + \overline B = (A_0 - B)^\circ \sqcap (A_1 - B)^\circ + \overline B = \bigl( (A_0 - B)^\circ + \overline B \bigr) \sqcap \bigl( (A_1 - B)^\circ + \overline B \bigr) \sqsupseteq A_0 \sqcap A_1$.

  For the third item, let us check that $A^\circ_B \sqsubseteq_o B$.
  Indeed,
  \begin{align*}
    (A^\circ_B - B)^\circ + \overline B
    &
    = \Bigl[ \bigl[ \bigl( (A-B)^\circ + \overline B \bigr) \sqcap A \bigr] - B \Bigr]^\circ + \overline B
    \\ &
    \sqsupseteq \Bigl[ (A-B)^\circ \sqcap (A-B) \Bigr]^\circ + \overline B
    \\ &
    = (A-B)^\circ + \overline B \sqsupseteq A^\circ_B.
  \end{align*}
  Assume now that $C \sqsubseteq_o B$ and $C \sqsubseteq A$.
  Then
  \begin{gather*}
    A^\circ_B
    = \bigl( (A-B)^\circ + \overline B \bigr) \sqcap A
    \sqsupseteq \bigl( (C-B)^\circ + \overline B \bigr) \sqcap C
    = C.
  \end{gather*}
\end{proof}

One should be careful, though: for example, $A \sqsubseteq_o B \sqsubseteq_o C$ does not imply that $A \sqsubseteq_o C$, and $A \sqsubseteq B \sqsubseteq C$, $A \sqsubseteq_o C$ does not imply $A \sqsubseteq_o B$.

\begin{lem}
  \label{lem:NiceOpen}
  Assume that $A \sqsubseteq_o B$, and let $U = (A-B)^\circ$.
  Then $B + U \sqsubseteq A$ and $B+U \sqsubseteq_o B$.
  Moreover, $B+U$ is dense in $A$, and if $A$ is dense in $B$ then $B+U$ is dense in $B$.
\end{lem}
\begin{proof}
  Clearly $B+U \sqsubseteq A$, and $B+U \sqsubseteq_o B$ since $U$ is open.
  By \autoref{lem:InequalityInteriorClosure1}, and since $A \sqsubseteq_o B$, we have $\overline{B+U} \sqsupseteq U^\circ + \overline B \sqsupseteq A$, so $B+U$ is dense in $A$, and if $A$ is dense in $B$ then so is $B+U$.
\end{proof}

\begin{dfn}
  Let $B \sqsubseteq X$, and let $O_n \sqsubseteq_o B$ for every $n \in \bN$.
  Then we say that $\bigsqcap_n O_n$ is $G_\delta$ in $B$.
  If each $O_n$ is moreover dense in $B$ then any intermediate grey set $\bigsqcap O_n \sqsubseteq A \sqsubseteq B$ is said to be \emph{co-meagre} in $B$.
\end{dfn}

When $B = X$ (i.e., when $B = \bzero_X$), this agrees with our previous definitions.
More generally,

\begin{lem}
  \label{lem:RelativeCoMeagre}
  Let $A \sqsubseteq B \sqsubseteq X$.
  \begin{enumerate}
  \item If $A$ is co-meagre in $B$ then $B \sqsubseteq^* A$ (equivalently, $A =^* B$).
  \item When $B$ is open, the converse holds as well.
  \end{enumerate}
\end{lem}
\begin{proof}
  Assume first that $A$ is co-meagre in $B$.
  By \autoref{lem:NiceOpen} we may assume that $A = B + \bigsqcap U_n$, where each $U_n$ is open, $B+U_n$ is dense in $B$, and $U_n \sqsupseteq U_{n+1}$.
  Fix $\varepsilon > 0$, and consider the ordinary open sets $V_n = (U_n)_{<\varepsilon}$ and $W_n = (X \setminus V_n)^\circ$.
  Since $\overline{C\rest_{W_n}} = \overline C \rest_{W_n}$ for any grey set $C$, we have
  We then have
  \begin{gather*}
    \overline B\rest_{W_n} = \overline{B + U_n}\rest_{W_n} \geq \overline{B+\varepsilon}\rest_{W_n} = \overline B\rest_{W_n} + \varepsilon.
  \end{gather*}
  Thus $\overline B$ and therefore $B$ and $A$ are infinite on $W_n$.
  Since $V_n \cup W_n$ is dense in $X$, the set $S = \bigcap (V_n \cup W_n)$ is co-meagre.
  If $x \in S$, then either $x \in W_n$ for some $n$, in which case $A(x) = B(x) = \infty$, or $x \in V_n$ for all $n$, in which case $U_n(x) = 0$ for all $n$ and $A(x) = B(x)$ as well, proving the first item.

  For the second, assume that $B$ is open, and that $\{x \in X : A(x) = B(x)\} \supseteq S = \bigcap V_n$ where $V_n \subseteq X$ are open and dense.
  Since $B$ is open, $B + \bzero_{V_n}$ is dense in $B$, so $B + \bzero_S$ is co-meagre in $B$, and \textit{a fortiori} so is $A$.
\end{proof}

We can now (state and) prove the Baire Category Theorem for grey sets.
For the proof, one essentially figures how to prove the classical Baire Category Theorem inside a $G_\delta$ subset of a Polish space without first proving that such a subset is a Polish space itself, and transfers the argument (modulo a reduction based on \autoref{lem:NiceOpen}) to the grey setting.

\begin{thm}
  \label{thm:Baire}
  Assume that $(X,d)$ is a complete metric space and $A \sqsubseteq B \sqsubseteq X$, where $B$ is $G_{\delta}$ (in $X$) and $A$ is co-meagre in $B$.
  Then $A$ is dense in $B$.
\end{thm}
\begin{proof}
  We have $B = \bigsqcap O_n$ with $O_n \sqsubseteq_o X$, and we may assume that $A = \bigsqcap A_n$ where each $A_n$ is open and dense in $B$.
  By virtue of \autoref{lem:NiceOpen}, we may assume that $A_n = B + U_n$ with $U_n \sqsubseteq_o X$.
  We may further assume that $O_n = \bigsqcap_{k \leq n} O_k$ and $U_n = \bigsqcap_{k \leq n} U_k$, in which case $A = \bigsqcap_n O_n + U_n$.

  To show that $A$ is dense in $B$, we apply the criterion presented in \autoref{lem:RelativeDensity}.
  Indeed, let $V \sqsubseteq_o X$, and it will be enough to show that if $\inf (B + V) < r$ then $\inf (A + V) < r$.

  Let us now construct by induction a sequence of $V_n \sqsubseteq_o X$ such that $\inf (V_n + B) < r$, starting with $V_0 = V$.
  Given $V_n$, and since $A_n$ is dense in $B$, we have $r > \inf (A_n + V_n)$, so there is $x_n \in X$ such that $r > (A_n + V_n)(x_n) \geq ( O_n + U_n + V_n )(x_n)$.
  For some $M_n > 0$ big enough, define $V_{n+1} = V_n + M_n d(\cdot,x_n)$, which is open as well.
  Then $r > (A_n + V_n)(x_n) = (A_n + V_{n+1})(x_n) \geq (B + V_{n+1})(x_n)$, and the induction may continue.

  Since $O_n + U_n + V_n$ is open, choosing $M_n$ big enough, we have $(O_n + U_n + V)(y) \leq (O_n + U_n + V_n)(y) < r$ whenever $d(y,x_n) < 2r/M_n$.
  We also observe that $d(x_{n+1},x_n) < r/M_n$, so choosing $M_{n+1} > 2M_n$, the sequence $\{x_n\}$ is Cauchy, converging to some $y$ such that $d(y,x_n) < 2 r/M_n$ for all $n$.
  Then $(A + V)(y) < r$, as desired.
\end{proof}

As explained in the introduction, the primary reason why we considered grey sets is that they are useful in the topometric setting.
We refer the reader to \cite{BenYaacov-Berenstein-Melleray:TopometricGroups} for details on topometric spaces and groups.
We simply recall here that a Polish topometric space is a triplet $(Y,\tau,\partial)$ such that $(Y,\tau)$ is a Polish space, $\partial$ is a metric refining $\tau$, and $\partial$ is $\tau$-lower semi-continuous.
We shall follow the convention that in the context of a topometric space $(Y,\tau,\partial)$, the vocabulary of general topology refers to the topology $\tau$, while the vocabulary of metric spaces refers to $\partial$ (unless qualified explicitly otherwise).

We now formulate a generalisation of the classical Kuratowski-Ulam Theorem to this context.
We consider a map $\pi\colon X \rightarrow Y$, where $X$ is a Polish space and $Y$ a topometric space.
For $A \sqsubseteq X$ and $y \in Y$ we need to define the fibres of $X$ and of $A$ over $y$, in a manner which takes into account the topometric structure on $Y$.

\begin{ntn}
  \label{ntn:DistanceThinckening}
  Let $(X,\partial)$ be a metric space.
  For a grey subset $A \sqsubseteq X$ we define the \emph{$\partial$-thickening} $(A)_\partial \sqsubseteq X$ as the greatest $1$-Lipschitz function lying below $A$, i.e.,
  \begin{gather*}
    (A)_\partial(x) = \inf_{x' \in X} \, A(x') + \partial(x,x'),
    \qquad
    (A)_{\partial<r} = \{x : \exists x' \, A(x') + \partial(x,x') < r\}.
  \end{gather*}
  When $A \subseteq X$ is an ordinary set, which we identify with its zero-indicator, we obtain $(A)_\partial = \partial(\cdot,A)$ and $(A)_{\partial<r} = \{x : \partial(x,A) < r\}$.
\end{ntn}

In particular, $A \sqsubseteq (X,\partial)$ is $1$-Lipschitz if and only if $A = (A)_\partial$.

\begin{ntn}
  \label{ntn:MapFibreImage}
  Let $X$ be a topological space, $(Y,\partial)$ a metric space, and $\pi\colon X \rightarrow Y$ a map.
  For $A \sqsubseteq X$ and $y \in Y$ we define the \emph{fibre} of $A$ over $y$, denoted $A_y \sqsubseteq A \sqsubseteq X$, and the \emph{image} of $A$, denoted $\pi A \sqsubseteq Y$, by
  \begin{gather*}
    A_y(x) = A(x) + \partial(\pi x,y),
    \qquad
    (\pi A)(y) = \inf_{\pi x = y} A(x).
  \end{gather*}
\end{ntn}

In particular $X_y(x) = \partial(\pi x,y)$.
Combining the two, $(\pi A)_\partial \sqsubseteq Y$ is the ``infimum over the fibre'', namely,
\begin{gather*}
  (\pi A)_\partial(y) = \inf_x A(x) + \partial( \pi x,y ) = \inf_X \, A_y.
\end{gather*}
We remark that if $\partial$ is the ``trivial'' $0/\infty$ distance then we recover the usual definition of the fibre of a set $A \subseteq X$ over $y$, namely $A_y = A \cap \pi^{-1} y$ (identifying a set with its zero-indicator).
In this case our result below specialises to the classical Kuratowski-Ulam Theorem (formulated for an open map between Polish spaces rather than just for a projection).

\begin{dfn}
  \label{dfn:TopometricOpenMap}
  Let $(X,\tau_X)$ be a topological space, $(Y,\tau_Y,\partial_Y)$ a topometric space.
  We say that a map $\pi \colon X \to Y$ is \emph{topometrically open} if for every open $U \subseteq X$ (equivalently, every $U \sqsubseteq_o X$) we have $(\pi U)_\partial \sqsubseteq_o Y$, i.e., $(\pi U)_{\partial<r} \subseteq Y$ is open for every $r > 0$.

  We say that a topometric space $(Y,\tau,\partial)$ is \emph{adequate} if $\id\colon (Y,\tau) \rightarrow (Y,\tau,\partial)$ is topometrically open.
\end{dfn}

When $\partial$ is trivial this is the same as being open, and in particular $\id\colon (Y,\tau) \rightarrow (Y,\tau,\partial)$ is automatically topometrically open.
In the general case, however, the latter does not always hold.
One checks that:

\begin{lem}
  \label{lem:TopometricOpen}
  Let $(X,\tau,\partial)$ be a topometric space.
  Then the following are equivalent:
  \begin{enumerate}
  \item $X$ is adequate, i.e., the map $\id\colon (X,\tau) \rightarrow (X,\tau,\partial)$ is topometrically open.
  \item For every $1$-Lipschitz $A \sqsubseteq X$, the closure $\overline A$ is $1$-Lipschitz as well.
  \item Same, with interior instead of closure.
  \end{enumerate}
\end{lem}

\begin{thm}
  \label{thm:KuratowskiUlam}
  Let $(Y,\tau,\partial)$ be an adequate Polish topometric space, $X$ a Polish space, and $\pi \colon X \to Y$ a continuous map (in the topology $\tau$, by our convention).
  Assume that $\pi\colon X \rightarrow Y$ is topometrically open, as per \autoref{dfn:TopometricOpenMap}.
  Then the following conditions are equivalent, for a Baire-measurable $A \sqsubseteq X$:
  \begin{enumerate}
  \item \label{item:KuratowskiUlam1} The grey set $A$ is co-meagre in $X$.
  \item \label{item:KuratowskiUlam2} The set $\{y \in Y: A_y \text{ is co-meagre in } X_y \}$ is co-meagre in $Y$.
  \end{enumerate}
\end{thm}
\begin{proof}
  \begin{cycprf}
  \item[\impnext]
    This easily reduces to the case when $A$ is open dense in $X$, so we only consider that case.
    Clearly, $A_y$ is open in $X_y$.
    For $U \sqsubseteq_o X$, define
    \begin{gather*}
      \Omega_U = \bigl( \pi (A + U) \bigr)_\partial - (\pi U)_\partial \sqsubseteq Y,
      \qquad \text{i.e.,} \qquad
      \Omega_U(y) = \inf (A_y + U) -  \inf (X_y + U).
    \end{gather*}
    By our assumption, $\Omega_U$ is the difference of two open grey sets, and is therefore $G_\delta$.
    Consider now any $V \sqsubseteq_o Y$.
    By assumption $(V)_\partial$ is open in $Y$ as well, and so $\pi^{-1} (V)_\partial = (V)_\partial \circ \pi$ is open in $X$.
    Since $A$ is dense,
    \begin{gather*}
      \inf_Y \, \Bigl( \bigl( \pi (A + U) \bigr)_\partial + V \Bigr)
      = \inf_X \, \bigl( A + U + \pi^{-1}(V)_\partial \bigr)
      = \inf_X \, \bigl( U + \pi^{-1}(V)_\partial \bigr) = \inf_Y \bigl( (\pi U)_\partial + V \bigr).
    \end{gather*}
    Thus $\bigl( \pi (A + U) \bigr)_\partial$ is dense in $( \pi U )_\partial$, and it follows that $\Omega_U$ is dense in $Y$.
    Since $X$ has a countable base, $\Omega = \bigsqcap_U \Omega_U$ is co-meagre, and $\Omega_{=0} \subseteq \{y \in Y : A_y \text{ is dense in } X_y \}$.
    Thus the latter is a co-meagre set.
  \item[\impfirst]
    Assume that $A$ is Baire-measurable, satisfies the hypothesis, and yet is not co-meagre, so $U(A)$ is not dense, and neither is $\half U(A)$.
    In other words, there is $V \sqsubseteq_o X$ such that $\inf V < \inf \bigl( V + \half U(A) \bigr)$, and we may assume that $\inf V = 0 < r < \inf \bigl( V + \half U(A) \bigr)$.
    Let $B = U(A) \dotminus A$, which is co-meagre, so by the implication we have already established, the set $C = \{y: A_y \text{ and } B_y \text{ are co-meagre in } X_y\}$ is co-meagre.
    If $y \in C$ then $\half[ A_y + B_y ]$ is co-meagre in $X_y$ as well (or, equivalently, $A_y + B_y$ is co-meagre in $2 X_y$).

    Since $(\pi V)_\partial \sqsubseteq_o Y$, we have $\inf_C (\pi V)_\partial = \inf_Y (\pi V)_\partial = 0$.
    Therefore there exists $y \in C$ such that $r > (\pi V)_\partial(y) = \inf_X (V + X_y)$.
    Since $X_y$ is $G_\delta$ (being closed), $\half[ A_y + B_y ]$ is dense in $X_y$, so
    \begin{gather*}
      r > \inf_X \bigl( V + \half[ A_y  + B_y ] \bigr) \geq \inf_X \bigl( V + \half U(A) \bigr) > r,
    \end{gather*}
    a contradiction.
  \end{cycprf}
\end{proof}

Let us mention an immediate application of this theorem.

\begin{prp}
  \label{prp:CoMeagreEnlargement}
  Let $(X,\tau,\partial)$ be an adequate Polish topometric space.
  Assume that $A \sqsubseteq U \sqsubseteq X$, where $U$ is open and $A$ is co-meagre in $U$.
  Then $(A)_\partial$ is co-meagre in $(U)_\partial$.

  In particular, if $A \sqsubseteq X$ is $1$-Lipschitz (relative to $\partial$), then $U(A)$ is also $1$-Lipschitz.
\end{prp}
\begin{proof}
  Let $B = A - U$.
  Then $B$ is co-meagre in $X$, which means in particular that it is Baire-measurable.
  By \autoref{thm:KuratowskiUlam} applied to $\id \colon X \to X$, the set $C = \{x \in X : B_x \text{ is co-meagre in } X_x\}$ is co-meagre in $X$.
  For $x \in C$ we have
  \begin{gather*}
    (U)_\partial(x) = \inf U_x = \inf (U + X_x) = \inf (U + B_x) = \inf A_x = (A)_\partial(x).
  \end{gather*}
  By \autoref{lem:RelativeCoMeagre}, $(A)_\partial$ is co-meagre in $(U)_\partial$.

  Now, assume that $A \sqsubseteq X$ is $1$-Lipschitz, equivalently $A=(A)_\partial$.
  By definition of $U(A)$ we have $U(A) = ^* A \sqcap U(A)$.
  By \autoref{lem:RelativeCoMeagre}, $A \sqcap U(A)$ is co-meagre in $U(A)$ and by the above $\bigl( U(A) \sqcap A \bigr)_\partial$ is co-meagre in $\bigl( U(A) \bigr)_\partial$.
  By \autoref{lem:RelativeCoMeagre} again,
  \begin{gather*}
    \bigl( U(A) \bigr)_\partial \sqsubseteq^* \bigl( U(A) \sqcap A \bigr)_\partial \sqsubseteq (A)_\partial = A.
  \end{gather*}
  By definition of $U(A)$, we obtain $\bigl( U(A) \bigr)_\partial \sqsubseteq U(A)$.
  Equality ensues, and $U(A)$ is $1$-Lipschitz.
\end{proof}

Foregoing the grey terminology, the above proposition says that, if $(X,\tau,\partial)$ is a Polish topometric space such that $(U)_{\partial<\varepsilon}$ is open for any open $U$ and any $\varepsilon >0$, and $A \subseteq X$ is co-meagre in some open $U \subseteq X$, then $(A)_{\partial<r}$ is co-meagre in $(U)_{\partial<r}$ for all $r$.
This fact is crucial for the proof of the topometric generalisation of Effros' theorem presented at the end of this article, and we do not know how to prove it without using the machinery of grey sets, in particular the grey version of the Kuratowski--Ulam theorem.

\section{Grey subsets of completely metrisable groups}

Throughout this section, let $G$ denote a completely metrisable topological group. We now introduce two operations on grey subsets; the operation $*$ reminds one of convolution.

\begin{dfn}
  \label{dfn:GreyStar}
  For two grey subsets $A,B \sqsubseteq G$ we define $A^{-1}, A * B \sqsubseteq G$ by
  \begin{gather*}
    A^{-1}(x) = A(x^{-1}), \qquad
    A * B(x) = \inf_{yz = x} \bigl( A(y) + B(z) \bigr).
  \end{gather*}
\end{dfn}
Note that $*$ is associative and has $\bzero_{\{1_G\}}$ as a neutral element.
We observe that for $A,B \subseteq G$, $\bzero_A^{-1} = \bzero_{A^{-1}}$ and $\bzero_A * \bzero_B = \bzero_{A\cdot B}$.
Thus, ${}^{-1}$ and $*$ extend the group operations of $G$, applied to subsets, to grey subsets (and should be thought of as operations on subsets, rather than as operations on group elements).
As expected, we have, for all grey subsets $A,B$ of $G$, that
\begin{gather*}
  \left(A * B \right)^{-1} = B^{-1} * A^{-1}.
\end{gather*}
By extension, for $x \in G$ we define $xA = \bzero_{\{x\}} * A$, namely $xA(y) = A(x^{-1}y)$, so $x \bzero_A = \bzero_{x\cdot A}$, and similarly for $Ax$.
We then obtain
\begin{gather}
  \label{eq:GreyStar}
  A * B = \bigsqcup_x \bigl( A(x) + xB \bigr) = \bigsqcup_x \bigl( Ax + B(x) \bigr).
\end{gather}
\begin{lem}
  \label{lem:GreyStarOpen}
  If $A\sqsubseteq_o G$ and $B \sqsubseteq G$ then $A^{-1} \sqsubseteq_o G$, $A * B \sqsubseteq_o G$ and $B * A \sqsubseteq_o G$.
\end{lem}
\begin{proof}
  The first assertion is obvious, the others are by \autoref{eq:GreyStar} above.
\end{proof}

\begin{prp}[Pettis Theorem for grey subsets]
  \label{prp:Pettis}
  Let $A,B \sqsubseteq G$ be grey subsets.
  Then $U(A) * U(B) \sqsubseteq A * B$.
\end{prp}
\begin{proof}
  We first observe that $U(A^{-1}) = U(A)^{-1}$, $U(xA) = xU(A)$, $U(Ax) = U(A)x$, by continuity of the group operations.
  Let $x \in G$, and let $D_x(y) = U(A)(y) + U(B)(y^{-1}x)$.
  Then $D_x$ is open, and $D_x \sqsubseteq^* A + xB^{-1}$.
  It follows that $\inf D_x \geq \inf (A + xB^{-1})$, i.e., $U(A) * U(B)(x) \geq A * B(x)$, as desired.
\end{proof}

\begin{lem}
  \label{lem:NonmeagreBaireMeasruableGreySubset}
  Let $A \sqsubseteq G$ be a non meagre Baire-measurable grey subset.
  Then $(A * A^{-1})^\circ(1) = 0$.
\end{lem}
\begin{proof}
  Let $B = U(A)$.
  Since $A$ is Baire-measurable and non meagre, we have $\inf B = 0$, thus $B * B^{-1}(1) = 2\inf B = 0$.
  By \autoref{prp:Pettis} we have $B * B^{-1} \sqsubseteq A * A^{-1}$, and since $B * B^{-1} \sqsubseteq_o G$, we obtain $B * B^{-1} \sqsubseteq (A * A^{-1})^\circ$.
  Thus $(A * A^{-1})^\circ(1) = 0$.
\end{proof}

The following is just a translation of the definition of a subgroup as a set to grey sets.

\begin{dfn}
  A \emph{grey subgroup} of $G$ is a grey subset $H \sqsubseteq G$ satisfying the following properties:
  \begin{itemize}
  \item $\inf H = 0$,
  \item $H * H^{-1} \sqsubseteq H$, i.e., $H(x) + H(y) \geq H(x y^{-1})$.
  \end{itemize}
\end{dfn}

Assume that $H$ is a grey subgroup.
It follows that $H(1) \le \inf_{x y = 1} H(x) + H(y^{-1}) = 2 \inf H = 0$.
Applying this to the second property we get $H^{-1} = 1 * H^{-1} \sqsubseteq H * H^{-1} \sqsubseteq H$, and therefore $H^{-1} = H = H * H$.
Finally, if $H \sqsubseteq G$ is a grey subgroup then so is $\overline H$, since then $\inf \overline H \leq \inf H  = 0$ and
\begin{gather*}
  \overline H(x) + \overline H(y) = \liminf_{x'\to x, y' \to y} H(x') + H(y') \geq \liminf_{x'\to x, y' \to y} H(x'y'^{-1}) = \overline H(xy^{-1}).
\end{gather*}
The same argument with $\limsup$ works for $H^\circ$, with the caveat that $\inf H^\circ$ need not be zero.

\begin{lem}
  \label{lem:OpenGreySubgroup}
  Let $H \sqsubseteq G$ be a grey subgroup.
  Then the following are equivalent:
  \begin{enumerate}
  \item $\inf H^\circ = 0$ (equivalently, $H^\circ$ is a grey subgroup).
  \item $H$ is open in $G$.
  \item $H$ is clopen in $G$.
  \end{enumerate}
\end{lem}
\begin{proof}
  \begin{cycprf}
  \item[\impnext]
    If $H^\circ$ is a grey subgroup then $H^\circ(1) = 0$.
    It follows that $H = H * H \sqsupseteq H * H^\circ \sqsupseteq H + 0 = H$, so $H = H * H^\circ$ is open.
  \item[\impnext]
    Assume $H(x) > r$ for some $x \in G$ and $r \in \bR$.
    Then $U = H_{<H(x)-r}$ is open, and if $z \in U$ then $H(xz) \geq H(x) - H(z) > r$, so $x \in xU \subseteq H_{>r}$.
    Thus $H$ is closed as well and therefore clopen.
  \item[\impfirst]
    Since $H = H^\circ$.
  \end{cycprf}
\end{proof}

\begin{rmk}
  \label{rmk:SemiNormDistance}
  What we call a grey subgroup of $G$ is usually called a \emph{semi-norm} on $G$.
  For a grey subgroup $H$ and a left-invariant pseudo-metric $d$ on $G$ (where we allow infinite distance) define
  \begin{gather*}
    d_H(g,h) = H(g^{-1}h), \qquad H_d(g) = d(1,g).
  \end{gather*}
  Then $d \mapsto H_d$ and $H \mapsto d_H$ are inverses, yielding a natural bijection between grey subgroups and left-invariant pseudo-metrics on $G$.

  Notice also that $H \sqsubseteq G$ is closed (open) if and only if $d_H$ is, and more generally, $\overline{d_H} = d_{\overline H}$, and if $H^\circ$ is a grey subgroup then $d_H^\circ = d_{H^\circ}$.
\end{rmk}

\begin{lem}
  \label{lem:NonmeagreBaireMeasruableGreySubgroup}
  Let $H \sqsubseteq G$ be a non meagre Baire-measurable grey subgroup.
  Then $H$ is clopen.
\end{lem}
\begin{proof}
  By \autoref{lem:NonmeagreBaireMeasruableGreySubset}, $(H * H^{-1})^\circ(1) = 0$, i.e., $H^\circ(1) = 0$, so $H$ is clopen by \autoref{lem:OpenGreySubgroup}.
\end{proof}

In usual terminology, this lemma says that if a semi-norm $H$ on a completely metrisable topological group $G$ is Baire-measurable and for any $\varepsilon >0$ the set of $g$ such that $H(g) \le \varepsilon$ is non meagre, then $H$ must be continuous.

\section{The small density property}
\label{sec:SmallDensity}

\begin{dfn}
  Let $G$ be a Polish group, $H \sqsubseteq G$ a grey subgroup.
  Let $d_H$ be the corresponding left-invariant pseudo-metric, as per \autoref{rmk:SemiNormDistance}.
  Then we define the \emph{index} $[G:H]$ to be the density character of $d_H$, namely the least cardinal of a $d_H$-dense subset of $G$.
\end{dfn}

We observe that if $H$ is an ordinary subgroup then this agrees with the usual definition of index.
Here we are going to be interested in the condition $[G:H] < 2^{\aleph_0}$ (i.e., ``$H$ has small index'').
Since the cofinality of $2^{\aleph_0}$ is uncountable, the following are equivalent:
\begin{enumerate}
\item $[G:H] < 2^{\aleph_0}$.
\item For all $\varepsilon > 0$, $G$ can be covered by fewer than continuum many left translates of $H_{<\varepsilon}$.
\item For all $\varepsilon > 0$, a family of disjoint left translates of $H_{<\varepsilon}$ has cardinal smaller than the continuum.
\end{enumerate}

\begin{prp}
  \label{prp:AutomaticContinuityBaireSmallIndex}
  Let $G$ be a completely metrisable group, and $H$ be a Baire-measurable grey subgroup of $G$.
  Assume also that $[G:H]< 2^{\aleph_0}$.
  Then $H$ is clopen.
\end{prp}
\begin{proof}
  If $G$ is not a perfect topological space then it is discrete and there is nothing to prove, so we assume that $G$ has no isolated points.
  In view of \autoref{lem:NonmeagreBaireMeasruableGreySubgroup}, it is enough to show that our assumptions imply that $H$ is not meagre.
  Assume for a contradiction that there exists $r>0$ such that $H_{<r}$ is meagre.
  Since the map $(x,y) \mapsto x^{-1} y$ is surjective and open from $G^2$ to $G$, we obtain that $\{(x,y) : H(x^{-1}y) < r \}$ is meagre in $G^2$.
  Then the Kuratowski--Mycielski Theorem implies that there exists a Cantor set $K \subseteq G$ such that, for any $x \neq y \in K$, one has $H(x^{-1}y) \ge r$ (see \cite{Kuratowski:IndependentSets} for the general version of this theorem that we use here, which is valid in any completely metrisable perfect space).
  This contradicts our assumption on the index of $H$, and we are done.
\end{proof}

Recall from \cite{BenYaacov-Berenstein-Melleray:TopometricGroups} that a \emph{Polish topometric group} is a triplet $(G,\tau,\partial)$, where $(G,\tau)$ is a Polish group, $\partial$ is a bi-invariant metric refining $\tau$ and $\tau$-lower semi-continuous.
The canonical example one should have in mind when thinking of this is the isometry group of some Polish metric space $(X,d)$ (or, more generally, the automorphism group of some Polish metric structure), endowed with the topology of pointwise convergence and the supremum metric $\partial(g,h)= \sup \{d(gx,hx) : x \in X \}$.
Note that any Polish topometric group $(G,\tau,\partial)$ is adequate as a Polish topometric space.
Indeed, if $U \subseteq G$ is an open set then
\begin{gather*}
  (U)_{\partial<r}= \bigcup_{\partial(g,1) < r} gU
\end{gather*}
is open as well.

We also need the following fact.

\begin{prp}
  Let $(G,\tau,\partial)$ be a Polish topometric group.
  Then $\partial$ is a complete metric on $G$.
\end{prp}
\begin{proof}
  Denote by $\partial_u$ a metric generating the coarsest bi-invariant uniformity refining $\tau$; it is well-known that $\partial_u$ is complete whenever $G$ is a Polish group.
  Since the uniformity generated by $\partial$ must be finer than the one generated by $\partial_u$, we may assume without loss of generality that $\partial \ge \partial_u$.
  Let $(g_n)$ be a $\partial$-Cauchy sequence in $G$.
  Then $(g_n)$ is $\partial_u$-Cauchy in $G$, hence it must $\partial_u$-converge to some $g \in G$.
  Pick $\varepsilon >0$, and let $N$ be such that for any $n,m \ge N$ one has $\partial(g_n,g_m) \le \varepsilon$.
  Fixing $n$ and letting $m$ go to $+ \infty$ we obtain, since $\partial$ is $\tau$-lower semicontinuous, that $\partial(g_n,g) \le \varepsilon$ for all $n \ge N$.
\end{proof}

\begin{dfn}
  We say that a Polish topometric group $(G,\tau,\partial)$ has the \emph{small density property} if whenever $H \sqsubseteq G$ is a $\partial$-Baire-measurable grey subgroup of index $< 2^{\aleph_0}$ then $H$ is open.
\end{dfn}

\begin{rmk}
  We do not call this property the small index property, because even when $\partial$ is the discrete metric the small density property as defined above is stronger than the usual small index property (which corresponds to left-invariant \emph{ultra-metrics} rather than left-invariant metrics).
\end{rmk}

We recall that a Polish topometric group $(G,\tau,\partial)$ admits ample generics if, for any integer $n$, there exists $\bar g \in G^n$ such that $\bigl(\{ k\bar g k^{-1} \colon k \in G\}\bigr)_\partial$ is co-meagre in $G^n$ (see \cite{BenYaacov-Berenstein-Melleray:TopometricGroups} for a discussion of this definition).

\begin{thm}
  \label{thm:SmallDensity}
  Let $(G,\tau,\partial)$ be a Polish topometric group admitting ample generics.
  Then $G$ has the small density property.
\end{thm}
\begin{proof}
  Let $H \sqsubseteq G$ be a $\partial$-Baire-measurable grey subgroup of small index, and let us show that $H$ is clopen.
  By \autoref{prp:AutomaticContinuityBaireSmallIndex}, $H$ is $\partial$-clopen.

  We recall that
  \begin{gather*}
    (H)_\partial(x) = \inf_y \bigl( H(y) + \partial(x,y) \bigr) = H * \partial = \partial * H,
  \end{gather*}
  where for $H * \partial$ we identify $\partial$ with the corresponding norm $\partial(1,\cdot)$.
  Accordingly, we define $(H)_{n\partial}$ with $n\partial$ in place of $\partial$, observing that $(H)_{n\partial}$ is also a grey subgroup, and $(H)_{n\partial} \sqsupseteq H$, whereby $H * (H)_{n\partial} = (H)_{n\partial}$.
  Also, $\bigsqcap_n (H)_{n\partial} = \overline H^\partial = H$.

  By the Kuratowski--Mycielski Theorem, $H$ cannot be meagre, that is, $H_{<\varepsilon}$ is not meagre for any $\varepsilon > 0$.
  Assume, for a contradiction, that for some $\varepsilon > 0$ the set $B = G \setminus \bigl[ (H)_\partial \bigr]_{<\varepsilon}$ is non meagre in every open subset of $G$, and let $A = H_{<\varepsilon/4}$.
  Below, for $x \in G$ and $C \subseteq G$ we denote by $C^x$ the set $x^{-1} C x$.
  By \cite[Lemma~3.6]{BenYaacov-Berenstein-Melleray:TopometricGroups} we can find a mapping $a \in 2^\omega \mapsto h_a \in G$ such that if $a,b \in 2^\omega$ are distinct then $\partial\bigl( A^{h_a},B^{h_b} \bigr) < \varepsilon/4$, i.e., $\partial\bigl( A^{h_ah_b^{-1}},B \bigr) < \varepsilon/4$.
  It follows that $H(h_ah_b^{-1}) \geq \varepsilon/4$ for all $a \neq b$, so $[G:H] = 2^{\aleph_0}$, a contradiction.

  Thus, for all $\varepsilon > 0$, the set $\bigl[ (H)_\partial \bigr]_{<\varepsilon}$ is co-meagre in some non empty open set, so $\inf U\bigl( (H)_\partial \bigr) = 0$.
  By Pettis' Theorem, we have $U\bigl( (H)_\partial \bigr) * U\bigl( (H)_\partial \bigr) \sqsubseteq (H)_\partial * (H)_\partial = (H)_\partial$, so $\inf (H)_\partial^\circ = 0$ and $(H)_\partial$ is a clopen grey subgroup.
  Since $(H)_\partial \sqsupseteq (H)_{n\partial} \sqsupseteq n(H)_\partial$, the latter implies $(H)_{n\partial}^\circ(1) = 0$ for all $n$, so $(H)_{n\partial}$ is a clopen grey subgroup for all $n$.
  Since $H = \bigsqcap_n (H)_{n\partial}$, $H$ is closed.
  Since $H$ is also non meagre, $\inf H^\circ = 0$, so $H$ is clopen.
\end{proof}

Notice that \cite[Theorem~4.7]{BenYaacov-Berenstein-Melleray:TopometricGroups} follows from \autoref{thm:SmallDensity} via a straightforward reduction to the case where the target group is metrisable, giving a more elegant proof than the one appearing there.
Conversely, \autoref{thm:SmallDensity} also follow from a combination of \autoref{prp:AutomaticContinuityBaireSmallIndex} with \cite[Theorem~4.7]{BenYaacov-Berenstein-Melleray:TopometricGroups}.

%It may also be worthwhile to translate this result in the more common language of left-invariant metrics.
%Assume that $(G,\tau,\partial)$ is a Polish topometric group with ample generics, and let $\rho$ be a left-invariant pseudo-metric on $G$.
%Let $G_{\rho}$ denote the metric space obtained by identifying points $g,g'$ such that $\rho(g,g')=0$, and assume that the density character of $G_{\rho}$ is strictly less than the continuum.
%The theorem above says that, under these assumptions, $\rho$ must be continuous (with regard to the topology $\tau$ of $G$) as soon as each set
%$\{g : \rho(g,1) \le r\}$ is $\partial$-closed. Even without this assumption on the level sets of $\rho$, one can obtain some information: indeed, the function $\rho'$ defined by
%\[
%\rho'(g,h)= \liminf_{g' \xrightarrow[\partial]{} g, h' \xrightarrow[\partial]{} h} \rho(g',h')
%\]
%is still a left-invariant pseudo metric on $G$, and the density character of the associated metric space is still less than the continuum. The $\liminf$ ensures that the associated grey subgroup is $\partial$-closed, so our theorem says that $\rho'$ must be continuous with respect to the topology $\tau$.

\section{A remark concerning grey subgroups of automorphism groups}

We conclude our discussion of grey subgroups with a natural situation in which open grey subgroups arise.
We assume that the reader of this section is familiar with the formalism of continuous logic (\cite{BenYaacov-Usvyatsov:CFO,BenYaacov-Berenstein-Henson-Usvyatsov:NewtonMS}).

Let $\cM$ be a metric structure, $G = \Aut(\cM)$.
If $\cM$ is a classical structure then for every member $a \in M^{eq}$, the stabiliser $G_a$ is an open subgroup, and if $\cM$ is $\aleph_0$-categorical then every open subgroup is the stabiliser of some (real or imaginary) element.
In the metric case, on the other hand, the stabiliser $G_a$ is usually \emph{not} open in $G$, and again we encounter the need to consider grey subgroups.

\begin{dfn}
  Let $\cM$ be a metric structure, $G = \Aut(\cM)$.
  The \emph{(grey) stabiliser} of $a \in M$, still denoted $G_a$, is defined by $G_a(g) = d(a,ga)$.
\end{dfn}

The grey stabiliser is, by definition, an open grey subgroup, and the exact stabiliser is $G_{a,\leq 0}$.
As in classical logic (e.g., \cite[Theorem~1.6]{Ahlbrandt-Ziegler:QuasiFinitelyAxiomatisable}), the converse holds for $\aleph_0$-categorical structures.

\begin{prp}
  Let $\cM$ be an $\aleph_0$-categorical separable structure, $G = \Aut(\cM)$, and let $H \sqsubseteq_o G$ be a real-valued open grey subgroup (namely, we exclude $+\infty$ from the range).
  Then there exists an imaginary $a \in M^{eq}$ such that $H = G_a$.
\end{prp}
\begin{proof}
  We may consider $M^\bN$ as a sort, equipped with the distance
  \begin{gather*}
    d(b,c) = \bigvee_n 2^{-n} \wedge d(b_n,c_n).
  \end{gather*}
  We fix $a$ in this sort which enumerates a dense subset of $\cM$.
  By homogeneity, the set of realisations of $p = \tp(a)$ in $\cM$ is exactly $\overline{Ga}$.
  We observe that $\{G_{a,<\varepsilon}\}_{\varepsilon > 0}$ is a base of neighbourhoods of the identity, so for every $\varepsilon > 0$ there exists $\delta(\varepsilon) > 0$ such that $G_{a,<\delta(\varepsilon)} < H_{<\varepsilon}$.

  Let $\varepsilon > 0$ and $g,g',h,h' \in G$, and assume that $d(ga,g'a)$ and $d(ha,h'a)$ are smaller than $\delta = \delta(\varepsilon/2)$.
  Then $g^{-1}g', h^{-1}h' \in G_{a,<\delta} \subseteq H_{<\varepsilon/2}$, whereby $|d_H(g,h) - d_H(g',h')| < \varepsilon$.
  Thus the map $\varphi\colon Ga \times Ga \to \bR$ sending $(ga,ha) \mapsto d_H(g,h)$ is well defined and uniformly continuous, and thus extends uniquely to a uniformly continuous function $\varphi\colon \overline{Ga} \times \overline{Ga} \to [0,\infty]$.
  By the Ryll-Nardzewski Theorem \cite[Fact~1.14]{BenYaacov-Usvyatsov:dFiniteness}, since $\varphi$ is uniformly continuous and invariant under automorphism, it is a definable pseudo-metric on the set defined by $p$ (and therefore in particular bounded).
  By \cite{BenYaacov:DefinabilityOfGroups}, and since by $\aleph_0$-categoricity the set defined by $p$ is definable, $\varphi$ extends to a definable pseudo-metric on all of $M^\bN$.
  (To recall the argument, by the Tietze extension theorem we may extend $\varphi$ to something definable on all of $M^\bN \times M^\bN$, call it $\varphi_0(x,y)$, and then  $\varphi_1(x,y) = \sup_{z \vDash p} |\varphi_0(x,z) - \varphi_0(y,z)|$ is a definable pseudo-metric which agrees with $\varphi$ on $p$.)
  Finally, with $[b]$ denoting the canonical parameter for $\varphi(x,b)$,
  \begin{gather*}
    H(g) = d_H(g,1) = \varphi(ga,a) = d\bigl( [a], [ga] \bigr).
  \end{gather*}
  Thus $H$ is precisely the grey stabiliser of $[a]$.
\end{proof}

\section{Topometric versions of two classical theorems}

A basic but important tool when studying Polish group actions is Effros' theorem \cite{Effros:TransformationGroups}: whenever $G$ is a Polish group acting continuously on a Polish space $X$, an element $x \in X$ has a co-meagre orbit if and only if $G  x$ is dense and the map $g \mapsto g  x$ is an open map from $G$ to $G  x$.
Combined with a theorem of Hausdorff, this implies immediately that $G  x$ must then be a dense $G_{\delta}$.
In particular, whenever a Polish group has a co-meagre conjugacy class, this conjugacy class is a dense $G_{\delta}$.
It is of interest to extend these results of Effros and Hausdorff to the topometric context, which we manage to do below, under the assumption that our topometric spaces are adequate.
In the proof of \autoref{thm:EffrosTopometric}, the attentive reader will notice the crucial use of \autoref{prp:CoMeagreEnlargement}, which requires the use of grey sets.
Even though this proposition seems to be used in a trivial case, it appears (insofar as we can see) indispensable for the proof.

Let us first recall the notion of a strong Choquet game as well as some facts regarding it.
While Kechris \cite{Kechris:Classical} defined the strong Choquet game for a separable metrisable space $A$ working entirely inside $A$ (denoted $X$ there), for our purposes it will be convenient to embed $A$ is some Polish space $X$.
Then the strong Choquet game $G_A$ is as an infinite game where two players $I$ and $I\!I$ take turns to play.
At step $i$:
\begin{itemize}
\item Player $I$ plays an open set $U_i$ contained in $V_{i-1}$ (with $V_{-1}=X$) and a point $x_i \in U_i \cap A$.
\item Player $I\!I$ plays an open set $V_i$ containing $x_i$ and contained in $U_i$.
\end{itemize}
Player $I\!I$ wins the game if $\bigcap U_i$ (which is equal to $\bigcap V_i$) intersects $A$.
By \cite[Theorem~8.17]{Kechris:Classical}, Player $I\!I$ has a winning strategy in $G_A$ if and only if $A$ is $G_\delta$ in $X$ (i.e., $A$ is Polish).

When $A \subseteq B \subseteq X$ define a game $G_{A,B}$ which follows the same rules, only that Player $I\!I$ wins if $\bigcap U_i$ intersects $B$.
Following the proof of \cite[Theorem~8.17]{Kechris:Classical}, Player $I\!I$ has a winning strategy in $G_{A,B}$ if and only if there exists a $G_\delta$ set $C$ such that $A \subseteq C \subseteq B$.

\begin{thm}[Topometric Hausdorff Theorem]
  \label{thm:Hausdorff}
  Let $Y$ be a Polish space, $(X,\tau,\partial)$ be a Polish topometric space, and $F \colon Y \to X$ be a continuous map.
  Assume moreover that $F\colon Y \rightarrow C$ is topometrically open for some $F(Y) \subseteq C \subseteq \overline{F(Y)}^\partial$.
  Then there exists $C \subseteq D \subseteq \overline{F(Y)}^{\partial}$ which is $G_{\delta}$ in $X$.

  In particular, under our hypotheses $\overline{F(Y)}^{\partial}$ is co-meagre in $\overline{F(Y)}^{\tau}$, and if $F\colon Y \rightarrow \overline{F(Y)}^\partial$ is topometrically open then $\overline{F(Y)}^\partial$ is $G_\delta$ in $X$.
\end{thm}
\begin{proof}
  It will suffice to prove that for every $\varepsilon > 0$, $I\!I$ has a winning strategy in $G_{C,B}$ where $B = \bigl( F(Y) \bigr)_{\partial\le\varepsilon}$.
  We assume without loss of generality that $I$ only plays sequences of open sets with vanishing diameter with respect to some complete compatible distance on $X$ such that $\overline U_i \subseteq V_{i-1}$.

  The strategy will produce a sequence of open subsets $W_i \subseteq Y$ and $I\!I$ will always play open sets $V_i \subseteq U_i$ such that $C \cap V_i = C \cap U_i \cap \bigl( F(W_i) \bigr)_{\partial<\varepsilon}$ (and $W_{-1} = Y$).
  Assume we are at turn $i$ of the game, and $I$ has just played $(U_i,x_i)$, so in particular $x_i \in V_{i-1} \subseteq \bigl( F(W_{i-1}) \bigr)_{\partial<\varepsilon}$.
  Pick $y_i \in W_{i-1}$ such that $\partial\bigl( F(y_i), x_i \bigr) < \varepsilon$, and choose $W_i$ open, containing $y_i$ and such that $\overline W_i \subseteq W_{i-1}$, arranging that $W_i$ have vanishing diameter with respect to some complete compatible distance on $Y$.

  Then $y_i \rightarrow y$ and $\bigcap W_i = \{y\}$ in $Y$, while $x_i \rightarrow x$ and $\bigcap U_i = \{x\}$ in $X$.
  Since $\partial$ is $\tau$-lower semicontinuous we have $\partial\bigl( F(y), x \bigr) \le \varepsilon$, so $x \in B$ and we are done.
\end{proof}

\begin{thm}[Topometric Effros Theorem]
  \label{thm:EffrosTopometric}
  Let $(X,\tau,\partial)$ be an adequate Polish topometric space and $G$ a Polish group acting continuously on $X$ by $\tau$-homeomorphisms and $\partial$-isometries.
  Assume that $x \in X$ is such that $G  x$ is dense.
  Then the following conditions are equivalent:
  \begin{enumerate}
  \item
    $\overline{G  x}^{\partial}$ is $G_{\delta}$.
  \item
    $\overline{G  x}^{\partial}$ is co-meagre.
  \item
    The map $G \rightarrow \overline{G  x}^\partial$, $g \mapsto gx$, is topometrically open.
    For any open subset $U$ of $G$ and any $r >0$, $(U  x)_{<r}$ is open in $\overline{G  x}^{\partial}$.
  \item
    There exists $y \in \overline{G  x}^{\partial}$ such that the map $G \rightarrow G  y$, $g \mapsto gy$, is topometrically open.
  \end{enumerate}
\end{thm}
\begin{proof}
  \begin{cycprf}
  \item Since $G  x$ is dense.
  \item
    Let $\pi \colon G \to X$ send $g \mapsto gx$.
    Fix a countable basis $(O_n)_{n< \omega}$ for the topology of $G$.
    Recall that a continuous image of a Borel subset of a Polish space is analytic and therefore Baire-measurable.
    Therefore, for any $n$ the function $(\pi O_n)_\partial$ is Baire-measurable and $1$-Lipschitz (relative to $\partial$).
    By \autoref{prp:CoMeagreEnlargement}, $U_n=U\bigl( (\pi O_n)_\partial \bigr)$ is also $1$-Lipschitz.
    Let $\Omega= \{y \colon \forall n (\pi O_n)_\partial(y)=U_n(y)\}$. This is a $\tau$-co-meagre, $\partial$-closed subset. Also, for any $O \sqsubseteq_o G$, $(\pi O)_{\partial} \sqcap \bzero_\Omega \sqsubseteq_o \bzero_\Omega$.

    Now, let $B=\{y \colon \forall^* g \in G \ g  y \in \Omega\}$. This set is $G$-invariant, $\tau$-co-meagre, and $\partial$-closed. The first point is obvious, the second follows from the Kuratowski-Ulam theorem, and to see why the third holds assume that $b_i \in B$ and $b \in X$ are such that $\partial(b_i,b) \to 0$. Then there exists $g \in G$ such that $g  b_i \in \Omega$ for all $i$, so since $\Omega$ is $\partial$-closed we get $g  b \in \Omega$, i.e.~$b \in \Omega$.

    It follows that $\overline{ G  x}^\partial$ is contained in $B$; to conclude, it is enough to prove that for all $U \sqsubseteq_o G$ $(\pi U)_\partial \sqcap \bzero_B \sqsubseteq_o \bzero_B$.
    To that end, let $b_i \in B$ converge to $b \in B$; there exists $g \in G$ such that $g b \in \Omega$ and $g b_i \in \Omega$ for all $i$.

    Since $(\pi gU)_\partial \sqcap \bzero_\Omega \sqsubseteq_o \bzero_\Omega$, we have $\limsup (\pi gU)_\partial(g  b_i) \le (\pi gU)_\partial(gb)$, equivalently $\limsup (\pi U)_\partial (b_i) \le (\pi U)_\partial (b)$.
  \item[\impfirst] By \autoref{thm:Hausdorff}.
  \item[\impnum{3}{4}] Obvious.
  \item[\impnum{4}{2}] By \autoref{thm:Hausdorff}, with $C = G  y$, the set $\overline{G  y}^{\partial}$ is co-meagre in $\overline{G  y}^{\tau}$.
    Since $\overline{G  y}^{\partial}=\overline{G  x}^{\partial}$, we obtain that $\overline{G  x}^{\partial}$ is co-meagre in $X$.
  \end{cycprf}
\end{proof}

Since any Polish topometric group is adequate as a topometric space, we obtain that in any Polish topometric group with ample generics the set of metric generic elements is $G_{\delta}$ in $G^n$ for all $n$.
It is natural to wonder whether $(G  x)_\partial$ being co-meagre is also equivalent to $(U  x)_\partial$ being open in $(G  x)_\partial$ for all $U$ open in $G$.
While one implication follows immediately from the above result, further development of ``grey topology'' is probably necessary in order to prove (or refute) the other.

\bibliographystyle{begnac}
\bibliography{begnac}

\end{document}